%% file: main.tex
\documentclass[british]{scrartcl}

\title {Segal \texorpdfstring{\emph{K}}{K}-theory of vector spaces\\
with an automorphism}
\author{Andrea Bianchi\and Florian Kranhold}

\input{preamble}

\lohead{A.\ Bianchi, F.\ Kranhold}
\rohead{Segal K-theory of vector spaces with an automorphism}

\begin{document}

\maketitle

\begin{abstract}
  We describe the Segal $K$-theory of the symmetric monoidal category of
  finite-dimensional vector spaces over a field $\bF$ together with an
  automorphism, or, equivalently, the group-completion of the $E_\infty$-algebra
  of maps from $S^1$ to the disjoint union of classifying spaces $\B\GL_d(\bF)$,
  in terms of the $K$-theory of finite field extensions of $\bF$. A key
  ingredient for this is a computation of the Segal $K$-theory of the category
  of finite-dimensional vector spaces with a nilpotent endomorphism. We also
  discuss the topological cases of $\bF=\C,\R$.\looseness-1
\end{abstract}

\section{Introduction and overview}
\label{sec:intro}
\input{1-introduction}

\section{Preliminaries}
\label{sec:prelim}
\input{2-preliminaries}

\section{Primary decompositions of automorphisms}
\label{sec:eigen}
\input{3-eigenspaces}

\section{Segal \texorpdfstring{$\mathbold{K}$}{K}-theory of nilpotent endomorphisms}
\label{sec:nil}
\input{4-nilpotent}

\section{The topological case}
\label{sec:top}
\input{5-topological}

\printbibliography[heading=bibintoc]

\addr{Andrea Bianchi}{%
  Max Planck Institute for Mathematics,\\
  Vivatsgasse 7,
  53111 Bonn,
  Germany,\\
  \mail{bianchi@mpim-bonn.mpg.de}.}

\addr{Florian Kranhold}{%
  Karlsruhe Institute of Technology,\\
  Englerstraße 2,
  76131 Karlsruhe,
  Germany,\\
  \mail{kranhold@kit.edu}.}

\end{document}

%% file: preamble.tex
\KOMAoptions{headings=standardclasses,numbers=enddot}

\usepackage[utf8]{inputenc}
\usepackage[T1]{fontenc}

\usepackage{babel}
\usepackage{csquotes}
\usepackage[cleanlook]{isodate}

\usepackage{lmodern}                                   %
\usepackage[osf,sc]{mathpazo}                          %
\newcommand{\LF}[1]{{\fontfamily{pplx}\selectfont #1}} %
\usepackage{myT1classico}                              %

\usepackage{setspace}
\usepackage{ellipsis} %
\usepackage[babel,tracking=true]{microtype}
\makeatletter
\@ifpackagelater{microtype}{2020/12/08}
  {\microtypesetup{nopatch=footnote}}
  {}
\makeatother
\UseMicrotypeSet[tracking]{smallcaps}
\SetTracking{encoding=*,shape=sc}{30}

\makeatletter
  \widowpenalty=\@M
\makeatother

\everypar=\expandafter{\the\everypar\looseness=-1}

\usepackage{relsize}
\newcommand{\acr}[1]{\texorpdfstring{\textsmaller{\LF{#1}}}{#1}}

\usepackage[shortlabels]{enumitem}

\usepackage{xcolor}

\IfFileExists{preamble-math.tex}  {\input{preamble-math.tex}}  {}
\IfFileExists{preamble-bib.tex}   {\input{preamble-bib.tex}}   {}
\IfFileExists{preamble-music.tex} {\input{preamble-music.tex}} {}
\IfFileExists{preamble-custom.tex}{\input{preamble-custom.tex}}{}

\usepackage{footmisc}

\addtokomafont{pagehead}{\fontfamily{ppl}\selectfont}

\usepackage{etoolbox}
\ifundef{\abstract}{}{
  \patchcmd{\abstract}{\quotation}{\quotation\noindent\ignorespaces}{}{}}

\usepackage{scrlayer-scrpage}
\usepackage[pdfusetitle,%
            bookmarksnumbered=true,%
            unicode,%
            hidelinks]{hyperref}

\let\C\COld

\let\mailfont\textsf
\newcommand{\mail}[1]{\upshape\href{mailto:#1}{\mailfont{#1}}}

\hypersetup{
  colorlinks,
  linkcolor={black!45!red},   %
  citecolor={black!55!green}, %
  urlcolor ={black!45!blue}   %
}

\IfFileExists{preamble-env.tex}{\input{preamble-env.tex}}{}

\newcommand{\addr}[2]{%
  \paragraph{#1}
  \LF{#2}\vspace*{-.25\baselineskip}
} 

\KOMAoptions{DIV=9}

%% file: preamble-math.tex
%
\usepackage{amssymb,amsmath} %
\usepackage{mathtools}       %
\usepackage{stmaryrd}        %

\usepackage{tikz-cd}

\let \ge       \geqslant

\let \setminus \smallsetminus

\let \epsilon  \varepsilon

\newcommand{\on} [1]{\mathrm{#1}}

\let\originalleft\left
\let\originalright\right
\renewcommand{\left}{\mathopen{}\mathclose\bgroup\originalleft}
\renewcommand{\right}{\aftergroup\egroup\originalright}

\newcommand{\ula}[1]{\langle #1\rangle}                                         %
\newcommand{\llb}[1]{\mathopen{}\left\llbracket #1\right\rrbracket\mathclose{}} %

\newcommand{\pa}[1]{(#1)}

\newcommand{\GL} {\on{GL}}

\newcommand{\op} {{\on{op}}}

\newcommand{\Id} {\on{id}}   %

\newcommand{\Mod}{\mathbf{Mod}}

\renewcommand{\Im}{\mathrm{Im}}

\newcommand{\N}{\mathbb{N}}
\newcommand{\Z}{\mathbb{Z}}
\newcommand{\Q}{\mathbb{Q}}
\newcommand{\R}{\mathbb{R}}
\newcommand{\C}{\mathbb{C}}

\DeclareFontFamily{OMS}{zplm}{
   \skewchar\font=48 %
   \fontdimen 9\font=0.34\fontdimen6\font
   \fontdimen 8\font=0.84\fontdimen6\font
}
\DeclareFontShape{OMS}{zplm}{m}{n}{<-> zplmr7y}{}
\DeclareFontShape{OMS}{zplm}{b}{n}{<-> zplmb7y}{}
\DeclareFontShape{OMS}{zplm}{bx}{n}{<->ssub * zplm/b/n}{}

\newcommand{\bA}{\mathbb{A}}
\newcommand{\bF}{\mathbb{F}}

\newcommand{\bL}{\mathbb{L}}

\newcommand{\caS}{\mathbf{S}}
\newcommand{\bfA}{\mathbf{A}}
\newcommand{\bfC}{\mathbf{C}}

\newcommand{\bfE}{\mathbf{E}}

\newcommand{\bfI}{\mathbf{I}}

\newcommand{\bfN}{\mathbf{N}}

\newcommand{\frm}{\mathfrak{m}}

\newcommand{\tF}{\tilde{F}}

\newcommand{\B}{\mathrm{B}}
\newcommand{\Map}{\mathrm{Map}}

\newcommand{\aut}{{\mathrm{aut}}}

\newcommand{\nil}{\mathrm{nil}}
\renewcommand{\top}{\mathrm{top}}

\newcommand{\tors}{\mathrm{tors}}
\newcommand{\Cat}{\mathbf{Cat}}
\newcommand{\CMon}{\mathbf{CMon}}
\newcommand{\Sp}{\mathbf{Sp}}
\newcommand{\hypo}{\mathrm{hypo}}
\newcommand{\BU}{\mathrm{BU}}
\newcommand{\BO}{\mathrm{BO}}
\newcommand{\mO}{\mathrm{O}}
\newcommand{\gl}{\mathrm{gp}}
\newcommand{\Tel}{\mathrm{Tel}}
\newcommand{\cn}{\mathrm{cn}}

\newcommand{\tto}{\mathbin{\mbox{$\to\hspace*{-3.4mm}\to$}}}

\makeatletter
\newcommand{\oset}[3][0ex]{%
  \mathrel{\mathop{#3}\limits^{
    \vbox to#1{\kern-2\ex@
    \hbox{$\scriptstyle#2$}\vss}}}}
\makeatother

\newcommand{\QK}{K}
\newcommand{\NF}{\bfN_{\smash\bF}}

%% file: preamble-bib.tex
\usepackage[backend      = biber,
            style        = alphabetic,
            maxbibnames  = 9,
            maxcitenames = 9,
            bibencoding  = utf8,
            datamodel    = ext-eprint]{biblatex}

\DefineBibliographyExtras{british}{%
  \renewcommand{\finalandcomma}{,}
}

\DeclareFieldFormat{doi}{\textsf{\href{https://doi.org/#1}{[DOI]}}}
\DeclareFieldFormat[misc]{title}{‘#1’}  %

\makeatletter
\DeclareFieldFormat{eprint:arxiv}{%
  arXiv\addcolon\space
  \ifhyperref
    {\href{http://arxiv.org/\abx@arxivpath/#1}{%
       \nolinkurl{#1}%
       \iffieldundef{eprintclass}
     {}
     {\addspace\textsf{\mkbibbrackets{\thefield{eprintclass}}}}}}
    {\nolinkurl{#1}
     \iffieldundef{eprintclass}
       {}
       {\addspace\textsf{\mkbibbrackets{\thefield{eprintclass}}}}}}%
\makeatother

\DefineBibliographyExtras{british}{\protected\def\bibrangedash{\nobreakdash--}}
\DefineBibliographyExtras{ngerman}{\protected\def\bibrangedash{\nobreakdash--}}

\AtBeginBibliography{\footnotesize}

\bibliography{bibliography.bib}

%% file: preamble-env.tex
%
\makeatletter
  \@ifpackageloaded{amsthm}{}{\usepackage{amsthm}}
\makeatother

\usepackage[noabbrev]{cleveref}

\theoremstyle{plain}
\newtheorem{theo}   {Theorem}[section]

\newtheorem{lem}    [theo]{Lemma}

\newtheorem{cor}    [theo]{Corollary}

\newtheorem{atheo}  {Theorem}

\theoremstyle{definition}
\newtheorem{defi}   [theo]{Definition}

\newtheorem{expl}   [theo]{Example}
\newtheorem{nota}   [theo]{Notation}
\newtheorem{constr} [theo]{Construction}
\newtheorem{rmk}    [theo]{Remark}
\newtheorem{rmd}    [theo]{Reminder}

\Crefname{atheo}{Theorem}{Theorems}

\newlength{\intendedDistance}

%% file: 1-introduction.tex
\subsection{Setting and question}
If $(\bfC,\otimes)$ is a symmetric monoidal category, the core groupoid
$\bfC^\simeq$ is canonically an $E_\infty$-algebra in spaces, whose
group-completion is the infinite loop space of a connective spectrum
$K^S(\bfC)$, known as the \emph{Segal $K$-theory of $\bfC$}
\cite{Segal,May78,Thomason79}.\looseness-1

We are interested in the relation between the Segal $K$-theories of $\bfC$ and
of the functor category $\bfC^\aut=\Map(\B\Z,\bfC)$ of objects in $\bfC$ with
automorphism, together with the pointwise symmetric monoidal structure.  Note
that $\bfC$ is a symmetric monoidal retract of $\bfC^\aut$, comprising objects
of $\bfC$ together with the identity morphism. As a consequence, the connective
spectrum $K^S(\bfC^\aut)$ contains $K^S(\bfC)$ as a direct summand, and we are
interested in determining the remaining summand. Moreover, the
$E_\infty$-algebra $\bfC^{\aut,\simeq}$ is equivalent to the free loop space
$\Map(S^1,\bfC^{\simeq})$, see \cref{lem:mapS1}, so our question can also be
phrased in terms of the group-completion of the latter.\looseness-1

When $\bfC$ is an abelian category with the direct sum monoidal structure,
$K^S(\bfC)$ agrees with the connective cover of the spectrum $K^\oplus(\bfC)$,
which we shall refer to as \emph{additive $K$-theory} of $\bfC$, and which is
defined as the (not necessarily connective) Quillen $K$-theory of $\bfC$ endowed
with the split exact structure.  We mention that, more generally, an additive
$K$-theory spectrum $K^\oplus(\bfC)$ can be defined for any small additive
$\infty$-category $\bfC$ by first taking the stabilisation of $\bfC$, which is a
small stable $\infty$-category, and by then taking $K$-theory of the latter in
the sense of \cite{BGT2}, see \cite[Appendix A]{ElmantoSosnilo},
\cite[Subsection 3.3]{LMMT}. The connective cover of the latter is equivalent to
$K^S(\bfC)$ \cite[Corollary 8.1.3]{HebestreitSteimle}.
  
On the other hand, each abelian category $\bfC$ can be endowed with its
\emph{maximal} exact structure; we denote its associated Quillen's $K$-theory by
$\QK(\bfC)$ and simply call it \emph{Quillen's $K$-theory} of $\bfC$ in the rest
of the article. By the theorem of the heart \cite{Barwick}, this agrees with the
$K$-theory of the bounded derived category of $\bfC$, which is a small, stable
$\infty$-category.

If $\bfC$ is abelian, then the aforementioned category $\bfC^\aut$ is again
abelian. In this work, we study $K^S(\bfC^\aut)$ and $K^\oplus(\bfC^\aut)$ in
the case where $\bfC=\Mod_\bF$ is the abelian category of finite-dimensional
vector spaces over a field $\bF$.\looseness-1

\subsection{Results}

We will see in \cref{cor:connective} that $K^\oplus(\Mod_\bF^\aut)$ is
connective, and hence agrees with Segal’s $K$-theory $K^S(\Mod_\bF^\aut)$. Hence
we do not have to distinguish between $K^S$ and $K^\oplus$ when stating our
results in the introduction.\looseness-1

As mentioned above, the $E_\infty$-algebra $\Mod_\bF^{\aut,\simeq}$ is
equivalent to the $E_\infty$-algebra of maps from $S^1$ to
$\Mod_\bF^{\simeq}\simeq \coprod_{d\ge 0}\B\GL_d(\bF)$. Using this description,
the main result of this article can be phrased as follows in terms of Segal
$K$-theory: For each field $\bF$, there is an equivalence of spectra
\begin{align}\label{eq:Main}%
  \B^\infty \Map\left(S^1,\coprod_{d\ge 0}\B\GL_d(\bF)\right)
  \simeq \bigoplus_{\substack{\pa{t}\ne \frm\subset
  \bF[t]\\\mathrm{maximal~ideal}}}\bigoplus_{i=1}^\infty
  K(\bF[t]/\frm).
\end{align}
In other words, the group completion of the $E_\infty$-algebra
$\Map\left(S^1,\coprod_{d\ge 0}\B\GL_d(\bF)\right)$ can be identified with the
infinite loop space of the right hand side.

Here $\bF[t]$ is a polynomial algebra in one variable, $\pa{t}$ is the ideal
generated by the monomial $t$, and $\bF[t]/\frm$ is the residue field of $\frm$.
Moreover, $K(\bF)$ is the usual algebraic $K$-theory spectrum of $\bF$, which
agrees with the formely introduced spectra
$\QK(\Mod_\bF)\simeq K^\oplus(\Mod_\bF)\simeq K^S(\Mod_\bF)$. Its homotopy type
has been determined in many cases, e.g.\ \cite{QuillenFin,Borel}, see also
\cite[§\,\textsc{vi}]{Weibel} for an overview.\looseness-1

The equivalence in \cref{eq:Main} is conceived in two main steps: In the first
step, carried out in \cref{sec:eigen}, we employ the fact that each automorphism
of a finite-dimensional vector space has a primary decomposition to get the
following result:\looseness-1

\begin{atheo}\label{thm:A}
  Let $\bF$ be a field. Then we have an equivalence of spectra
  \[K^\oplus(\Mod_\bF^\aut)\simeq\bigoplus_{\substack{\pa{t}\ne \frm\subset
        \bF[t]\\\mathrm{maximal~ideal}}} K^\oplus(\Mod^\nil_{\bF[t]/\frm}).\]
\end{atheo}

Here $\Mod^\nil_{\bF}$ is the abelian category of finite-dimensional vector
spaces together with a nilpotent endo\-morphism.  By admitting the formerly
excluded maximal ideal $(t)$, the equivalence of \cref{thm:A} can be extended
to\looseness-1
\[K^\oplus(\Mod_\bF^{\text{end}}) \simeq
  \bigoplus_{\substack{\frm\subset\bF[t]\\\text{maximal ideal}}}
  K^\oplus(\Mod^\nil_{\bF[t]/\frm}),\] where
$\Mod_\bF^{\text{end}}=\Map(\B\N,\Mod_\bF)$ is the abelian category of vector
spaces over $\bF$ together with an endomorphism, see \cref{rmk:Endos} for
details.\looseness-1

Under the equivalence of \cref{thm:A}, the summand $K^\oplus(\Mod_\bF)$ is a
retract of the summand
$K^\oplus(\Mod^\nil_{\smash{\bF[t]/\pa{t-1}}})\simeq K^\oplus(\Mod^\nil_{\bF})$;
this retract is obtained by combining the functors
$\smash{\Mod_\bF\to\Mod_{\smash\bF}^\nil}$, endowing a given vector space with
the zero endomorphism, and $\smash{\Mod_\bF^\nil\to \Mod_\bF}$, forgetting the
endomorphism.\looseness-1

In the second step, carried out in \cref{sec:nil}, we study the additive
$K$-theory of $\Mod^\nil_{\bF}$. Although the forgetful functor
$\Mod^\nil_\bF\to\Mod_\bF$ induces an equivalence on Quillen $K$-theory, the
same is far from being true for additive $K$-theory.  We instead show the
following:

\begin{atheo}\label{thm:B}
  Let $\bF$ be a field. Then we have an equivalence of spectra
  \[K^\oplus(\Mod_\bF^\nil)\simeq\bigoplus_{i=1}^\infty K(\bF).\]
\end{atheo}

Roughly speaking, the index set corresponds to the fact that nilpotent
endomorphisms can be decomposed into Jordan blocks of arbitrary size. Our proof
of \cref{thm:B} relies on an explicit description of the fibre of the comparison
map between $K^\oplus(\Mod^\nil_\bF)$ and $\QK(\Mod^\nil_\bF)$ in terms of
Quillen $K$-theory of a certain auxiliary abelian category; this description is
due to Auslander and Sherman \cite{Sherman} in the connective setting, and to
Schlichting \cite{Schlichting} in the non-connective setting used in this
work. The main equivalence (\ref{eq:Main}) then follows by combining
\cref{thm:A,thm:B}.\looseness-1

\begin{rmk}
  \label{rmk:thmAisKFlinear}
  If the ground field $\bF$ is assumed to be perfect, then the equivalences in
  \cref{thm:A,thm:B} are in fact equivalences of $K(\bF)$-module spectra: this
  follows at once by observing that all functors that are used for establishing
  the equivalences are $\bF$-linear and hence tensored over $\Mod_\bF$.  See
  \cref{rmk:nonperfect} for details.\looseness-1
\end{rmk}

\begin{rmk}
  The symmetric monoidal structure on $\Mod_\bF^\aut$ given by tensor product
  over $\bF$ gives rise to an $E_\infty$-ring structure on
  $K^\oplus(\Mod_\bF^\aut)$.  Given two maximal ideals
  $\frm,\frm'\in\mathrm{Spec}(\bF[t^{\pm1}])=\mathbb{G}_m(\bF)$, it is easy to
  see that the product of the two direct summands of $K^\oplus(\Mod_\bF^\aut)$
  corresponding to $\frm,\frm'$ lands inside the (finitely many) summands
  corresponding to maximal ideals in the image of the finite scheme
  \[\mathrm{Spec}(\bF[t^{\pm1}]/\frm\otimes_{\bF}\bF[t^{\pm1}]/\frm')\subset
    \mathbb{G}_m(\bF)\times_{\bF}\mathbb{G}_m(\bF)\]%
  along the multiplication map of $\mathbb{G}_m(\bF)$. It would be interesting
  to derive a complete description of the $E_\infty$-ring product in the light
  of the joint decomposition of the spectrum $K^\oplus(\Mod_\bF^\aut)$ given by
  \cref{thm:A,thm:B}.
\end{rmk}

In the final \cref{sec:top}, we discuss the topological cases of $\bF=\C$ and
$\bF=\R$, where the natural topology on the automorphism groups is taken into
account. More precisely, let $\Mod_{\C,\top}^\simeq$ and $\Mod_{\R,\top}^\simeq$
be the topologically enriched symmetric monoidal groupoids of finite-dimensional
vector spaces over $\C$ and $\R$, respectively, with isomorphisms as morphisms,
i.e.\ we have the two homotopy equivalences
\[\Mod_{\C,\top}^\simeq \simeq \coprod_{d\ge 0}\BU(d),\qquad
  \Mod_{\R,\top}^\simeq \simeq \coprod_{d\ge 0}\BO(d).\]
Very generally, for any $E_\infty$-algebra $A$ and any space $X$, there is a
canonical map of spectra $\B^\infty\Map(X,A)\to \Map(X,\B^\infty A)$. Note
that the source of the map is always connective, whereas the target may be
non-connective. For $A=\Mod_{\bF,\top}^\simeq$, we show:\looseness-3

\begin{atheo}
 \label{thm:C}
 For $\bF=\C,\R$ and for any finite cell complex\footnote{We will prove the
   corresponding $\infty$-categorical statement, which works for $X$ any compact
   anima.} $X$, the canonical map of
 spectra
 \[\B^\infty \Map(X,\Mod_{\bF,\top}^\simeq))
   \to\Map(X,\B^\infty \Mod_{\bF,\top}^\simeq)\]%
 is an equivalence after passing to connective covers.  Specialising to the case
 $X=S^1$, we have equivalences of spectra
 \begin{align*}
   \B^\infty\Map(S^1,\Mod_{\C,\top}^\simeq)&\simeq\Sigma^\infty_+S^1\wedge \mathrm{ku},\\
   \B^\infty\Map(S^1,\Mod_{\R,\top}^\simeq)&\simeq\mathrm{colim}\left(\mathrm{ko}\overset{r}{\longleftarrow}\mathrm{ku}\overset{r}{\longrightarrow}\mathrm{ko}\right).
 \end{align*}
\end{atheo}

Here $\mathrm{ku}$ and $\mathrm{ko}$ denote the connective complex and real
$K$-theory spectra, respectively, and $r\colon\mathrm{ku}\to\mathrm{ko}$ is the
‘realification’, induced by restriction of scalars $\C\to \R$. We point out that
the infinite loop space of the left hand side
$\B^\infty\Map(X,\Mod^\simeq_{\bF,\top})$ is exactly the group completion of the
$E_\infty$-algebra of vector bundles over $X$, recovering Atiyah's original
definition of topological $K$-theory of $X$. In the case of $X=S^1$, this
spectrum can be regarded as Segal’s $K$-theory of the topologically enriched
category of finite-dimensional vector spaces over $\bF$ together with an
automorphism.\looseness-1

The first statement of \cref{thm:C}, specialised to $X=S^1$, is in contrast with
\cref{thm:A,thm:B}, which imply that, for a discrete field $\bF$, the canonical map
\[K^\oplus(\Mod_\bF^\aut)\to \Map(S^1,K^\oplus(\Mod_\bF))\]%
is far from being an equivalence (even after passing to connective covers) as
the left-hand side is identified with the direct sum
$\bigoplus_{(t)\neq\frm\subset\bF[t]}\bigoplus_{i\ge1} K(\bF[t]/\frm)$, whereas
the right-hand side has the rather `smaller' homotopy type
$K(\bF)\oplus\Sigma^{-1} K(\bF)$.\looseness-1

There is, however, also a striking similarity between \cref{thm:C} and
\cref{thm:A}: We observe that the maximal ideals in $\C[t]$ different from
$\pa{t}$ naturally form a \emph{space}, namely $\C\setminus\{0\}\simeq S^1$, and
every maximal ideal $\frm$ yields the `same' residue field
$\C\cong\C[t]/\frm$. Similarly, maximal ideals in $\R[t]$ different from
$\pa{t}$ are parametrised by the space
$\{z\in\C\setminus\{0\}:\Im(z)\ge0\}\simeq[0,1]$. This space contains the
subspace $\R\setminus\{0\}\simeq\{0,1\}$, and the residue field $\R[t]/\frm$ is
isomorphic to $\R$ or to $\C$, depending on whether the parameter of $\frm$ lies
in $\R\setminus\{0\}$ or not. Finally, the lack of terms corresponding to Jordan
blocks of size at least $2$ is, intuitively, due to the fact that in the
topological setting every automorphism of a finite-dimensional vector space can
be canonically homotoped to a unitary or orthogonal one, which is
semisimple.\looseness-1

\subsection{Related work}\label{subsec:related}

The \emph{Quillen} $K$-theory of modules with an automorphisms, more precisely
the kernel of
$\QK_{\smash i\vphantom{x}}(\Mod^\aut_\bF)\to \QK_{\smash
  i\vphantom{x}}(\Mod_\bF)$, has been studied in \cite{Grayson1976,Grayson1979}.
Furthermore, \cite{Almkvist1}, describes $\QK_0$ of the category of modules with
an \emph{endomorphism}; the higher $K$-groups have been also discussed in
\cite{Almkvist2,Stienstra2}. In particular, we point out the similarity between
the decomposition of \cite[Thm.\,5.2]{Almkvist2} for Quillen $K$-theory and our
\cref{thm:A}. A more modern treatment of Quillen $K$-theory of endomorphisms can
be found in \cite[§\,3]{BGT}.

The additive $K$-theory of modules with an automorphism is a key ingredient in a
description of a weight filtration of Quillen $K$-theory of rings, due to
\cite{Grayson1995}. More precisely, Grayson considers, for a given ring $R$, a
simplicial ring $R\bA^\bullet$ whose $0$-simplices are $R$ itself. The
$n$\textsuperscript{th} filtration component on $\QK(R)$ is built out of the
Segal $K$-theories of $R\bA^\bullet$-modules together with at most $n$ mutually
commuting automorphisms. In that context, we calculated the $0$-simplices of the
first filtration step in the case where $R$ is a field.

The current article can also be regarded as an instance of the following
problem: Replacing $E_\infty$ by $E_1$ in the above discussion, the free loop
space\footnote{We focus on maps from $S^1$ for simplicity; the same problem can
  be reformulated with respect to maps from an arbitrary space.} $\Map(S^1,A)$
of an $E_1$-algebra $A$ has a pointwise $E_1$-algebra structure, admitting a
group-completion $\Omega \B\Map(S^1,A)$. We might also first group-complete $A$
and then take the free loop space, obtaining the space $\Map(S^1,\Omega \B
A)$. As before, there is a canonical $E_1$-map\looseness-1
\[\Omega \B \Map(S^1,A)\to \Map(S^1,\Omega \B A)\]
and we are broadly interested in understanding, in concrete examples, whether
this map is an equivalence, or rather how much this map is \emph{not} an
equivalence.  In \cite{Bianchi-Kranhold-Reinhold}, we determined, in a slightly
different setting, the homotopy type of the group-completion of $\Map(S^1,A)$,
where $A$ is the $E_1$-algebra given by the disjoint union of classifying spaces
of either of the following three sequences of groups: mapping class groups of
oriented surfaces of genus $g\ge 0$ with one boundary curve, Artin braid groups,
and symmetric groups. In all these cases, the above
map is far from being an equivalence.  This article studies the case of $A$
being the disjoint union of classifying spaces of general linear groups
$\GL_d(\bF)$.\looseness-1

\subsection*{Acknowledgements and funding}
We would like to thank Søren Galatius and Jens Reinhold for sparking our
interest in mapping spaces into $E_1$-algebras and their group-completion, and
for drawing our attention to the work of Grayson.  We furthermore owe thanks to
Lars Hesselholt and Bjørn Ian Dundas for making us aware of Auslander and
Sherman’s results. In addition to that, we are grateful to Simon Gritschacher,
Kaif Hilman, Manuel Krannich, and Christian Mauz for helpful conversations on
the subject. Finally, we would like to thank the anonymous referee for many
inspiring comments, which significantly improved the paper, including a way to
get around a superfluous assumption in \cref{sec:eigen}.\looseness-1

A.\,B. was partially supported by the \emph{European Research Council} under the
\emph{European Union’s Horizon 2020 research and innovation programme} (grant
agreement No.\ 772960),
by\kern-.18px~the\kern-.18px~\emph{Danish\kern-.18px~National\kern-.18px~Research\kern-.18px~Foundation}\kern-.18px~through\kern-.18px~the\kern-.18px~\emph{Copenhagen\kern-.18px~Centre\kern-.18px~for\kern-.18px~Geometry\linebreak
  and Topology} (\acr{DNRF}151), and by the Max Planck Institute for Mathematics
in Bonn.\looseness-1

%% file: 2-preliminaries.tex
\subsection{Segal \texorpdfstring{$\mathbold{K}$}{K}-theory}
\enlargethispage{\baselineskip}

We start by recalling the definition of Segal $K$-theory for symmetric monoidal
categories and some of its properties. While the most part of the paper can be
read and understood with Segal’s original construction \cite{Segal} in mind, we
decided to use the more modern language of $\infty$-categories from
\cite{Lurie1,Lurie2}.\looseness-1

\begin{nota}
  Let $\Cat_\infty$ and $\caS$ denote the (large) $\infty$-categories of small
  $\infty$-categories and small $\infty$-groupoids (a.k.a.\ \emph{spaces}), respectively. The
  $\infty$-categories of small symmetric monoidal $\infty$-categories and small
  symmetric monoidal $\infty$-groupoids (a.k.a.\ $E_\infty$-algebras in spaces)
  are denoted by $\CMon(\Cat_\infty)$ and $\CMon(\caS)$, respectively.
\end{nota}

\begin{rmd}
  The inclusion $\caS\hookrightarrow\Cat_\infty$ has a right-adjoint
  $({-})^\simeq\colon\Cat_\infty\to\caS$, called the \emph{core groupoid
    functor} \cite[Cor.\,5.1.17]{Land}. Since $({-})^\simeq$ preserves products,
  it induces a functor between categories of commutative monoid objects, which
  we shall still call $(-)^\simeq\colon\CMon(\Cat_\infty)\to\CMon(\caS)$. It is
  again right-adjoint to the forgetful functor
  $\CMon(\caS)\to \CMon(\Cat_\infty)$, see \cite[Lem.\,6.1]{GGN}.\looseness-1
\end{rmd}

\begin{rmd}
  The $\infty$-category of spectra is denoted by $\Sp$, and the full subcategory
  of connective spectra is denoted by by $\Sp^\cn$. The latter is equivalent to the full
  subcategory of $\CMon(\caS)$ spanned by group-like objects
  \cite[{}5.2.6.26]{Lurie2}, via the \emph{infinite loop space} functor
  $\Omega^\infty\colon\Sp^\cn\hookrightarrow\CMon(\caS)$. The functor
  $\Omega^\infty$ has a left-adjoint $\B^\infty\colon\CMon(\caS)\to\Sp^\cn$, see
  \cite{May1974} and \cite[Rmk.\,4.5]{GGN}.
\end{rmd}

Now we can give the definiton of Segal $K$-theory as in \cite[Def.\,8.3]{GGN}:

\begin{defi}\label{defi:SegalK}
  We define \emph{Segal K-theory} to be the composition of functors
  \[
    \begin{tikzcd}
      K^S\colon \CMon(\Cat_\infty)\ar[r,"(-)^\simeq"] &\CMon(\caS)\ar[r,"\B^\infty"]&\Sp^\cn.
    \end{tikzcd}
  \]
\end{defi}

\subsection{Mapping spaces and representations}

In this subsection, we make the aforementioned connection between categories of
automorphisms and free loop spaces precise.

\begin{constr}\label{constr:Fun}
  Let $\bfC$ be a symmetric monoidal $\infty$-category and let $\bfI$ be any
  small $\infty$-category.  Then the $\infty$-category $\Map(\bfI,\bfC)$ of
  functors from $\bfI$ to $\bfC$ is again symmetric monoidal by pointwise
  operations, and it is an $\infty$-groupoid if $\bfC$ is.

  In the case where both $\bfI$ and $\bfC$ are small $\infty$-groupoids, i.e.\
  spaces, the $\infty$-groupoid $\Map(\bfI,\bfC)$ can be regarded as a classical
  mapping space between spaces, which we also denote, more traditionally, as
  $\Map(\bfI,\bfC)$.\looseness-1
\end{constr}

\begin{expl}\label{ex:1cat}
  For a discrete monoid $G$, we denote by $\B G$ the $1$-category with a single
  object that has $G$ as its endomorphism monoid.
  
  If $\bfC$ is a (classical) category, then $\Map(\B G,\bfC)$ is the category
  of $G$-representations in $\bfC$; and if, additionally, $\bfC$ is symmetric
  monoidal, then the monoidal structure on $\Map(\B G,\bfC)$ is given by taking
  monoidal products of $G$-representations.

  For example, $\Map(\B\Z,\bfC)$ is the category whose objects are pairs
  $(X,\kappa)$ where $X$ is an object of $\bfC$ and $\kappa$ is an automorphism of
  $X$, and whose arrows $(X,\kappa)\to (X',\kappa')$ are morphisms
  $\phi\colon X\to X'$ with $\kappa'\circ\phi=\phi\circ\kappa$.
\end{expl}

\begin{nota}
  For any $\infty$-category $\bfC$, we write $\bfC^\aut\coloneqq \Map(\B\Z,\bfC)$.
\end{nota}

\begin{lem}\label{lem:mapS1}
  For any small symmetric monoidal $\infty$-category $\bfC$, we have an equivalence
  of connective spectra $K^S(\bfC^\aut)\simeq \B^\infty\Map(S^1,\bfC^\simeq)$. In
  particular, for a field $\bF$, we have an equivalence of connective spectra
  \[K^S(\Mod_\bF^\aut) \simeq \B^\infty\Map\left(S^1,\coprod_{d\ge 0}\B\on{GL}_d(\bF)\right).\]
\end{lem}
\begin{proof}
  Since $\B\Z\simeq S^1$ is an $\infty$-groupoid, the canonical functor of
  symmetric monoidal $\infty$-groupoids
  $\Map(S^1,\bfC^\simeq)\to \Map(S^1,\bfC)^\simeq$ is an equivalence by
  \cite[Rmk.\,5.1.18]{Land}. We hence have an equivalence of spectra\looseness-1
  \[
    K^S(\bfC^\aut)= \B^\infty \Map(S^1,\bfC)^{\simeq}\simeq
    \B^\infty \Map(S^1,\bfC^{\simeq}) =\B^\infty\Map(S^1,\bfC^\simeq).
  \]
  In the special case of $\bfC=\Mod_\bF$, we use that
  the $E_\infty$-algebra $\Mod_\bF^\simeq$ is equivalent to the disjoint union
  of classifying spaces $\coprod_{d\ge 0}\B\on{GL}_d(\bF)$.
\end{proof}

Thus, our original question can be phrased geometrically as follows: We want
to study, for a field $\bF$, the group-completion of the free loop space of
$\coprod_{d\ge 0}\B\on{GL}_d(\bF)$.

\subsection{Quillen \texorpdfstring{$\mathbold{K}$}{K}-theory and additive
  \texorpdfstring{$\mathbold{K}$}{K}-theory}

\begin{rmd}
  An \emph{exact structure} on an abelian category $\bfA$ is a class $\bfE$ of
  short sequences $A\to A'\to A''$ which satisfies a list axioms
  \cite[§2.1]{Quillen}. The pair $(\bfA,\bfE)$ is then called an \emph{exact
    category}. An \emph{exact functor} $(\bfA,\bfE)\to (\bfA',\bfE')$ is an
  additive functor $F\colon \bfA\to\bfA'$ with $F(\bfE)\subset \bfE'$. We
  refrain from recalling all axioms, as we will only be concerned with two
  classes of exact structures:\looseness-1
  \begin{itemize}
  \item The collection of all \emph{split} short exact sequences
    $A\to A\oplus A''\to A''$ is an exact structure on $\bfA$; it is actually
    the smallest one satisfying Quillen's axioms.  We refer to this as the
    \emph{split} exact structure on $\bfA$.
  \item The collection of \emph{all} short exact sequences is an exact structure
    on $\bfA$, which we refer to as the \emph{maximal} exact structure on
    $\bfA$. If not stated otherwise, this is the canonical exact structure that
    we consider on an abelian category.\looseness-1
  \end{itemize}
  For an exact category $(\bfA,\bfE)$, its Quillen $K$-theory has been defined
  as a loop space in \cite[§\,2.2]{Quillen}, further deloopings have been
  studied in \cite{Waldhausen,GilletGrayson,Jardine,Shimakawa}. In this work,
  the \emph{Quillen $K$-theory} $K(\bfA,\bfE)$ is a (not necessarily connective)
  spectrum, using the definition from \cite[§\,12.1]{Schlichting}.\looseness-1
\end{rmd}

\begin{nota}
  For an abelian category $\bfA$ we denote by $K^\oplus(\bfA)$ and $\QK(\bfA)$
  Quillen's $K$-theories of $\bfA$ with respect to the maximal and the split
  exact structures, respectively. We refer to the first as \emph{additive
    $K$-theory} of $\bfA$, and to the second, by abuse of terminology, simply as
  \emph{Quillen's $K$-theory} of $\bfA$.
\end{nota}

\begin{rmd}
  Any abelian category $\bfA$ can be regarded as a symmetric monoidal category
  via its (co-)cartesian structure.  Then it is a classical result that
  $K^S(\bfA)$ is naturally equivalent to the connective cover of
  $K^\oplus(\bfA)$, see \cite{Grayson1976a} for a proof on the level of spaces
  and \cite[Thm.\,10.2]{BohmannOsorno} and \cite[Rmk.\,9.33]{BGT2} for a proof
  on the level of spectra.
\end{rmd}

\begin{rmd}\label{rmd:noeth}
  We call an abelian category $\bfA$ \emph{noetherian} if for each object $A$ of
  $\bfA$, each ascending chain of subobjects
  $A_0\subset A_1\subset\dotsb\subset A$ eventually stops. It is shown in
  \cite[Thm.\,7]{Schlichting} that if $\bfA$ is noetherian, then $K(\bfA)$ is
  connective.
\end{rmd}

\begin{expl}
  If $\bF$ is any field, then the maximal and the split exact structure on the
  (small) abelian category $\Mod_\bF$ of finite-dimensional vector spaces over
  $\bF$ are equal. Therefore, the spectra $K^\oplus(\Mod_\bF)$ and
  $\QK(\Mod_\bF)$ are equivalent to each other. As $\Mod_\bF$ is noetherian,
  these spectra are connective, and hence agree with $K^S(\Mod_\bF)$. We shall
  denote their homotopy type simply by $K(\bF)$.
\end{expl}

\begin{rmd}\label{constr:comp}
  For a general abelian category $\bfA$, Quillen $K$-theory differs from
  additive $K$-theory, but there is a comparison map
  $\omega_A\colon K^\oplus(\bfA)\to \QK(\bfA)$, since the identity of $\bfA$ is
  an exact functor from the split to the maximal exact structure.
\end{rmd}

\subsection{Semiadditivity}

\begin{rmk}
  The $\infty$-category of (small) abelian categories and additive functors, and
  the $\infty$-category of spectra, are \emph{semiadditive}, i.e.\ initial and
  terminal objects exist and coincide, and for each finite collection
  $(X_i)_{i\in J}$ of objects, the canonical morphism from the coproduct
  $\coprod_{i\in J}X_i$ to the product $\prod_{i\in J}X_i$ is an equivalence,
  see e.g.\ \cite[Prop.\,2.3+8]{GGN}.
\end{rmk}

\begin{nota}
  For an arbitrary set $I$ and a family $(X_i)_{i\in I}$ of either abelian
  categories or spectra, we denote by
  $\bigoplus_{i\in I}X_i\coloneqq \coprod_{i\in I}X_i$ the categorical
  \emph{coproduct}, and we will refer to it as a `direct sum'.
\end{nota}

\begin{rmk}
  \label{rmk:coproductofproducts}
  In any ambient category, a coproduct $\coprod_{i\in I}X_i$ over an arbitrary
  set $I$ is the filtered colimit of the coproducts $\coprod_{i\in J}X_i$ over the poset of
  finite subset $J\subseteq I$. If, additionally, our ambient category is
  semiadditive, then we can replace each of the latter finite coproducts by a
  finite product and, accordingly, identify $\coprod_{i\in I}X_i$ with the
  filtered colimit of the finite \emph{products} $\prod_{i\in J}X_i$.
\end{rmk}

\begin{lem}\label{lem:dirSumAdd}
  Let $(\bfA_i)_{i\in I}$ be a family of abelian categories. Then the following
  canonical map of spectra is an equivalence:
  \[\bigoplus_{i\in I}K^\oplus(\bfA_i)\to K^\oplus\left(\bigoplus_{i\in I}\bfA_i\right).\]
\end{lem}
\begin{proof}
  By \cite[Cor.\,4+5]{Schlichting}, the functor $(\bfA,\bfE)\mapsto K(\bfA,\bfE)$ from the category of
  exact categories to spectra commutes with finite products and filtered
  colimits. Hence the statement follows from the observation that taking finite
  products and filtered colimits of split exact categories (in the category of
  exact categories) agrees with taking finite products and filtered colimits of
  the underlying abelian categories and endowing the result with its split exact
  structure.
\end{proof}

%% file: 3-eigenspaces.tex
The goal of this section is to prove \cref{thm:A}.
\enlargethispage{\baselineskip}

\begin{nota}
  Let $R$ be a commutative ring. For each $a\in R$, we denote by $Ra=\pa{a}$ the
  ideal generated by $a$, and we write $R/a\coloneqq R/\pa{a}$ for the quotient
  ring.
\end{nota}

\begin{rmd}[Primary decomposition]\label{rmd:prim}
  Let $R$ be a principal ideal domain and let $M$ be a finitely generated
  torsion module over $R$. Then there are maximal ideals
  $\frm_1,\dotsc,\frm_n\subset R$ and positive integers $r_1,\dotsc,r_n>0$, unique up to permutation, with
  \[M\cong \bigoplus_{i=1}^n R/\frm_i^{r_i}.\]%
\end{rmd}

\begin{expl}\label{ex:Ft}
  For a field $\bF$, we identify the category $\Mod_\bF^{\text{end}}$ of
  finite-dimensional vector spaces over $\bF$ together with an endomorphism,
  with the full abelian subcategory
  $\Mod^\tors_{\smash{\bF[t]}}\subseteq \Mod_{\smash{\bF[t]}}$ spanned by those
  $\bF[t]$-modules $M$ with $\dim_\bF M<\infty$: these are precisely the
  finitely generated torsion modules. In this description, the endomorphism of a
  vector space is given by multiplication by $t$ on a module.\looseness-1

  By the above primary decomposition theorem, we obtain a decomposition of each
  $M\in\Mod^\tors_{\smash{\bF[t]}}$ into summands of the form
  $\bF[t]/\frm_i^{\smash{r_i}}$, where each $\frm_i$ is generated by an
  irreducible polynomial. The endomorphism of $M$, given by multiplication by
  $t$, is an automorphism if and only if none of these ideals $\frm_i$ equals $\pa{t}$.
\end{expl}

\begin{defi}
  For a maximal ideal $\frm\subset\bF[t]$, we let $\Mod_{\smash{\bF[t]}}^\frm$
  be the full abelian subcategory of $\smash{\Mod_{\bF[t]}^\tors}$ containing
  all modules isomorphic to $\bigoplus_{i=1}^n\bF[t]/\frm^{r_i}$ for some
  $n\ge 0$ and $r_1,\dotsc,r_n>0$.
\end{defi}

\begin{expl}\label{ex:Nil}
  The category $\Mod_{\bF[t]}^{\smash{\pa{t}}}$ coincides with the abelian category
  of finite-dimensional vector spaces over $\bF$ together with a nilpotent
  endomorphism, also denoted by $\Mod_\bF^\nil$. This abelian category will play
  a central rôle in \cref{sec:nil}.\looseness-1
\end{expl}

\begin{lem}\label{lem:prim}
  We have an equivalence of abelian categories
  \[\Mod_\bF^\aut \simeq
    \bigoplus_{\substack{\pa{t}\ne \frm\subset \bF[t]\\\mathrm{maximal~ideal}}} \Mod_{\bF[t]}^\frm.\]
\end{lem}
\begin{proof}
  We have an additive functor
  $\bigoplus_\frm \Mod_{\bF[t]}^\frm\to \Mod_\bF^\aut$ given by taking a
  finitely supported family $(M_\frm)_{\frm}$ to the direct sum
  $\bigoplus_\frm M_\frm$. This functor is essentially surjective by the primary
  decomposition theorem, and it is clearly faithful.  It is also full, since for
  two different maximal ideals $\frm_1$ and $\frm_2$, the powers
  $\frm_1^{\smash{r_1}}$ and $\frm_2^{\smash{r_2}}$ are generated by coprime
  polynomials $q_1$ and $q_2$, and in such a situation, there are no non-zero
  $\bF[t]$-linear maps $\bF[t]/q_1\to \bF[t]/q_2$.
\end{proof}

Note that under the equivalence of \cref{lem:prim}, the retract
$\Mod_\bF\hookrightarrow \Mod_\bF^\aut$ from the introduction, given by endowing
a given vector space $V$ with the identity automorphism, is actually a retract
of the summand $\Mod_{\bF[t]}^{\smash{\pa{t-1}}}$, as the primary decomposition
of $(V,\Id_V)$ is given by $\bigoplus_{i=1}^{\smash{\dim(V)}}\bF[t]/\pa{t-1}$.
We can now study each summand $\Mod_{\bF[t]}^{\frm}$ separately.

\begin{lem}\label{lem:FExt}
  Let $\bF$ be a field and let $\frm\subset\bF[t]$ be any maximal ideal.  Then
  there is an equivalence of abelian categories
  \[\Mod_{\bF[t]}^\frm \simeq \Mod^\nil_{\bF[t]/\frm}.\]
\end{lem}

We thank the referee for suggesting the following argument, proving
\cref{lem:FExt} without our original assumption that
$\frm$ is generated by a separable polynomial.\looseness-1

\begin{proof}
  We denote by $\smash{\widehat{\bF[t]}{}^{\frm}}:=\lim_{n\in\N}\bF[t]/\frm^n$
  the completion of $\bF[t]$ at the ideal $\frm$. Then
  $\smash{\Mod_{\bF[t]}^\frm}$ is isomorphic to the abelian category of
  $\frm$-nilpotent and finitely generated modules over
  $\smash{\widehat{\bF[t]}{}^{\frm}}$. Similarly, $\Mod_{\bF[t]/\frm}^\nil$ is
  isomorphic to the category of $(x)$-nilpotent, finitely generated modules over
  $(\bF[t]/\frm)\llb{x}$, the ring of power series over $\bF[t]/\frm$. It is
  then a classical theorem of Cohen \cite{Cohen} that
  $\smash{\widehat{\bF[t]}{}^{\frm}}$, being a complete discrete valuation ring
  of equal characteristic, is isomorphic \emph{as a commutative ring} to
  $(\bF[t]/\frm)\llb{x}$. This isomorphism lets the (unique) maximal ideals
  correspond to each other, and hence gives rise to an isomorphism between
  categories of finitely generated modules that are nilpotent with respect to
  the maximal ideal.\looseness-1
\end{proof}

\begin{rmk}
  \label{rmk:nonperfect}
  If the maximal ideal $\frm$ from \cref{lem:FExt} is generated by a separable
  polynomial $q$, then the desired equivalence can be made very explicit: Let
  $\alpha$ be some root $q$ in the algebraic closure of $\bF$ and let
  $\bL\coloneqq \bF(\alpha)$ be the field extension of $\bF$ by $\alpha$, which
  is isomorphic to the residue field $\bF[t]/\frm$. If $\Mod^\frm_{\bL[t]}$
  denotes the full subcategory of $\Mod^\tors_{\bL[t]}$ spanned by those modules
  $M$ for which there exists an $r\ge 0$ with $q^rM =0$, then the composition of
  additive functors
  \[
    \begin{tikzcd}[column sep=6em]
      \Mod^\frm_{\bF[t]}\ar{r}{M\mapsto \bL[t]\otimes_{\bF[t]} M} & %
      \Mod^\frm_{\bL[t]}\ar{r}{M\mapsto M_{t-\alpha}} & \Mod^{(t-\alpha)}_{\bL[t]}
    \end{tikzcd}
  \]
  is an equivalence, where $M_{t-\alpha}\subseteq M$ denotes the
  $(t-\alpha)$-primary part of $M$. The target category, however, is equivalent
  to $\smash{\Mod^{(t)}_{\smash{\bL[t]}} = \Mod^\nil_{\bF[t]/\frm}}$ by a simple
  shift of coordinates. Since all these functors are $\bF$-linear, the
  equivalence from \cref{lem:FExt} does induce an equivalence of $K(\bF)$-module
  spectra on $K^\oplus$.

  However, if $\frm$ is not generated by a separable polynomial, then there
  exists in general no isomorphism between $\smash{\widehat{\bF[t]}{}^{\frm}}$
  and $(\bF[t]/\frm)\llb{x}$ \emph{as $\bF$-algebras}.  For instance, let $p$ be
  a prime number, let $\bF\coloneqq \bF_p(a)$ be the field of rational functions
  in a variable $a$, and let $\frm\coloneqq (t^p-a)\subset\bF[t]$. Then
  $\smash{\widehat{\bF[t]}{}^{\frm}}$ is not isomorphic as an $\bF$-algebra to
  $(\bF[t]/\frm)\llb{x}$, since the element $a\in\bF$ admits no
  $p$\textsuperscript{th} root in the first ring
  $\smash{\widehat{\bF[t]}{}^{\frm}}$: in fact $a$ has no
  $p$\textsuperscript{th} root even in the quotient ring $\bF[t]/\frm^2$; yet
  the `same' element $a$ admits a $p$\textsuperscript{th} root in
  $(\bF[t]/\frm)\llb{x}$. We also remark that the isomorphism of rings
  $\smash{\widehat{\bF[t]}{}^{\frm}}\simeq(\bF[t]/\frm)\llb{x}$ proved by Cohen
  is non-canonical, i.e.\ it depends on certain choices, when $\bF$ is
  non-perfect and $\frm$ is generated by a non-separable polynomial; on the
  contrary, if $\bF$ is perfect, there exists a unique (hence canonical)
  isomorphism of $\bF$-algebras
  $\smash{\widehat{\bF[t]}{}^{\frm}}\simeq(\bF[t]/\frm)\llb{x}$.\looseness-1

  As a consequence of these remarks, we will have the following: if $\bF$ is a
  perfect field, then the isomorphism from \cref{thm:A} will be a canonical
  isomorphism of $K(\bF)$-modules; but if $\bF$ is non-perfect, then we will
  only obtain a non-canonical isomorphism of spectra.
\end{rmk}

\begin{proof}[Proof of \cref{thm:A}]
  We use from \cref{lem:dirSumAdd} that $K^\oplus$ is compatible with direct
  sums of abelian categories, and combine \cref{lem:prim} and \cref{lem:FExt} to
  obtain
  \[K^\oplus(\Mod^\aut_\bF) \simeq
    \bigoplus_{\substack{\pa{t}\ne \frm\subset \bF[t]\\\mathrm{maximal~ideal}}}
    K^\oplus(\Mod^\frm_{\bF[t]}) \simeq
    \bigoplus_{\substack{\pa{t}\ne \frm\subset \bF[t]\\\mathrm{maximal~ideal}}}
    K^\oplus(\Mod^\nil_{\bF[t]/\frm}).\qedhere\]%
\end{proof}

\begin{rmk}
  The argument of the proof of \cref{thm:A} also works with respect to the
  maximal exact structure, recovering \cite[Thm.\,5.2]{Almkvist2}.
\end{rmk}

\begin{expl}
  If $\bF$ is algebraically closed, then all maximal ideals are of the form
  $(t-\alpha)$ with $\alpha\in\bF$, whence in this case \cref{thm:A}
  yields an equivalence of spectra
  \[K^\oplus(\Mod_\bF^\aut) \simeq \bigoplus_{\alpha\in\bF\setminus\{0\}}
    K^\oplus(\Mod_\bF^\nil).\]
\end{expl}

\begin{rmk}\label{rmk:Endos}
  Recall from \cref{ex:Ft} that the absence of primary summands of the form
  $\bF[t]/(t^r)$ is exactly what tells automorphisms from general endomorphisms.
  The equivalence of \cref{lem:prim} can hence be extended to an 
  equivalence between $\Mod_\bF^{\text{end}}\coloneqq \Map(\B\N,\Mod_\bF)$, the
  abelian category of vector spaces with an endomorphism, and
  $\bigoplus_{\frm\subset\bF[t]}\Mod^\frm_{\bF[t]}$, where the maximal ideal
  $(t)$ is now part of the indexing set. Since \cref{lem:FExt} can also applied
  in this situation, the equivalence of \cref{thm:A} can be generalised to the
  following equivalence of spectra
  \[K^\oplus(\Mod_\bF^{\text{end}})\simeq
    \bigoplus_{\substack{\frm\subseteq\bF[t]\\\text{maximal
          ideal}}}K^\oplus(\Mod_{\bF[t]/\frm}^\nil).\]
\end{rmk}

%% file: 4-nilpotent.tex
The goal of this section is to prove \cref{thm:B}, i.e.\ to determine the
homotopy type of the spectrum $K^\oplus(\Mod_\bF^\nil)$. Throughout this
section, we will abbreviate the abelian category $\Mod_\bF^\nil$ by $\NF$ to
save space.  In order to prove \cref{thm:B}, we once again employ the primary
decomposition theorem from \cref{rmd:prim}.\looseness-1

\begin{rmd}
  As in \cref{ex:Nil}, we identify $\NF$ with the full abelian
  subcategory of $\Mod_{\smash{\bF[t]}}$ containing all $\bF[t]$-modules $M$
  which satisfy $\dim_\bF M<\infty$ and $t^rM=0$ for some integer $r\ge 0$.

  For each integer $r\ge 0$, the $\bF[t]$-module $M_r\coloneqq \bF[t]/t^r$ is
  called \emph{Jordan block of size $r$}; we have $\dim_\bF M_r=r$, and the
  $\bF[t]$-linear morphisms $M_s\to M_r$ are all of the form
  $\phi^{r,s}_q\colon x\mapsto q\cdot x$ for some polynomial $q\in\bF[t]$, with
  the additional requirement that $t^{r-s}$ must divide $q$ if
  $s<r$.\looseness-1

  The primary decomposition theorem tells us that each object in $\NF$
  is isomorphic to a direct sum of the form $M_{r_1}\oplus\dotsb\oplus M_{r_n}$
  with $r_1,\dotsc,r_n\ge 1$.
\end{rmd}

We first recall the \emph{Quillen} $K$-theory of $\NF$, which can
easily be determined by \emph{dévissage}, a method due to Quillen that will turn
out to be useful later as well:

\begin{rmd}[Dévissage]
  Let $\bfA$ be an abelian category and let $\bfA'\subset\bfA$ be a full
  subcategory, which contains $0$ and is closed under isomorphisms, as well as
  taking subobjects, quotients and finite products in $\bfA$. Then $\bfA'$ is
  itself an abelian category and the inclusion functor
  $\bfA'\hookrightarrow\bfA$ is exact.
  
  We say that $\bfA'$ \emph{filters} $\bfA$ if for each object $A$ of $\bfA$
  there is a filtration
  \[0=A^0\subset\dotsb\subset A^l=A\]%
  such that all quotients $A^{k+1}/A^k$ lie in $\bfA'$. In this case, the
  dévissage theorem tells us that the induced map $\QK(\bfA')\to \QK(\bfA)$ is
  an equivalence, see \cite[Thm.\,4]{Quillen}, as well as
  \cite[Cor.\,5.4.6]{Mochizuki} for a version for (not necessarily connective)
  spectra.
\end{rmd}

\begin{cor}\label{cor:KQeasy}
  The functor $U\colon \NF\to \Mod_\bF$ that forgets the nilpotent
  endomorphism induces an equivalence on Quillen K-theory.
\end{cor}

This result is known a special case of Grayson’s fundamental theorem
\cite[p.\,236]{Grayson1976a}. In our setting, it is a consequence of the primary
decomposition theorem:\looseness-1

\begin{proof}
  The functor $U$ has an exact section $S$ by identifying $\Mod_\bF$ with the
  subcategory of $\NF$ spanned by vector spaces with the zero
  endomorphism; we prove that this section induces an equivalence on $\QK$. The
  subcategory of trivial endomorphisms is closed under isomorphisms, subobjects,
  quotients, and finite products, so we are left to show that it filters
  $\NF$. By the primary decomposition theorem, it suffices to show
  that each $M_i$ admits such a filtration. This follows inductively, using that
  $M_1$ carries the trivial endomorphism and using the short exact
  sequences\looseness-1
  \[\begin{tikzcd}
      0\ar[r] & M_{r-1}\ar{r}{\phi^{r,r-1}_t} & M_r\ar{r}{\phi^{1,r}_1} & M_1\ar[r] & 0.
    \end{tikzcd}\qedhere\]
\end{proof}

In order to understand $K^\oplus(\NF)$, we shall study the comparison map
between additive $K$-theory and Quillen $K$-theory from \cref{constr:comp}:

\begin{rmd}\label{rmd:Sherman}
  Let $\bfA$ be an abelian category. In \cite{Sherman}, the fibre of the
  comparison map $\omega_\bfA\colon K^\oplus(\bfA)\to \QK(\bfA)$ from
  \cref{constr:comp} has been identified with $\QK(\bfA^\#)$, where $\bfA^\#$ is
  the following abelian category:\footnote{In \cite{Sherman} the category
    $\bfA^\#$ is denoted as $\hat\bfA_0$, as it is conceived in several steps.}
  We denote the category of abelian groups by $\mathbf{Ab}$ and let
  $[-,-]\colon \bfA^\op\times\bfA\to \mathbf{Ab}$ be the Hom-functor. Then
  $\bfA^\#$ is the full subcategory of the abelian category of additive functors
  $\bfA^\op\to \mathbf{Ab}$, containing those functors $F$ such that there is an
  epimorphism $\beta\colon B\tto B'$ in $\bfA$ with\looseness-1
  \[F\cong \on{coker}([-,\beta]\colon [-,B]\to [-,B']).\]%
  It is shown in \cite[Prop.\,2.1]{Auslander} that $\bfA^\#$ is closed under
  taking biproducts, kernels, and cokernels in the functor category, and hence
  is an abelian subcategory. We note that a sequence $0\to F\to F'\to F''\to 0$
  in $\bfA^\#$ is exact if and only if $0\to FA\to F'A\to F''A\to 0$ is an
  exact sequence of abelian groups for all $A\in\bfA$.\looseness-1

  We point out that \cite{Sherman} only describes the underlying space of the
  fibre; an identification on the level of spectra can be found in
  \cite[Thm.\,9]{Schlichting}.
\end{rmd}

\begin{lem}\label{lem:split}
  We have an equivalence of spectra
  $K^\oplus(\NF)\simeq K(\bF)\oplus \QK(\NF^\#)$.
\end{lem}
\begin{proof}
  By naturality of the comparison maps $\omega_\bullet$ from \cref{constr:comp},
  we obtain a commuting square of spectra
 \[
    \begin{tikzcd}[column sep=3em]
      K^\oplus(\NF)\ar{r}{\omega_{\smash{\NF}}
      \vphantom{\omega_{A_A}}
      }
      \ar{d}[swap]{K^\oplus(U)} & \QK(\NF)\ar{d}{\QK(U)}[swap]{\simeq}\\
      K^\oplus(\Mod_\bF)\ar[r,"\omega_{\Mod_\bF}","\simeq"'] & \QK(\Mod_\bF),
    \end{tikzcd}
  \]
  where $U$ is the exact functor from \cref{cor:KQeasy}, having a
  section $S$. By functoriality, $K^\oplus(S)$ is a section of the left
  vertical map $K^\oplus(U)$; since the right vertical map $\QK(U)$ is an
  equivalence, we have that $K^\oplus(\NF)$ splits, up to equivalence, into
  $K^\oplus(\Mod_\bF)\simeq K(\bF)$ and the fibre of
  $\omega_{\smash{\NF}}$, which is
  $\QK(\NF^\#)$.\looseness-1
\end{proof}

It therefore remains to study $\QK$ of the abelian category $\NF^\#$.

\begin{lem}\label{lem:finite}
  Let $F$ be an object of $\NF^\#$. Then $\dim_\bF (FM_r)<\infty$ holds for each
  $r\ge 0$ and $FM_r=0$ for all but finitely many $r$. In particular, $\NF^\#$ is noetherian.
\end{lem}
\begin{proof}
  There is an epimorphism
  $\beta\colon N= \bigoplus_{i=1}^n M_{\smash{r_{\smash i}}}\tto N'=\bigoplus_{j=1}^m M_{\smash{s_{\smash j}}}$
  such that $F$ is isomorphic to the cokernel of $[-,\beta]$. In particular,
  $FM_r$ is a quotient of $[M_r,N']$ and therefore finite-dimensional.
  Moreover, if $r$ is larger than each of the $r_i$, then the map
  $[M_r,\beta]\colon[M_r,N]\tto[M_r,N']$ is again surjective,
  showing that $FM_r=0$.\looseness-1

  To show that each ascending filtration of each object $F$ becomes stationary,
  we note that for each subobject $F'\subset F$, either $F'=F$ holds or we have
  a strict inequality $\sum_{r\ge 0} \dim(F'M_r)<\sum_{r\ge 0} \dim(FM_r)$ of
  non-negative integers.
\end{proof}

\begin{cor}\label{cor:connective}
  The spectra $K^\oplus(\NF)=K^\oplus(\Mod_\bF^\nil)$ and
  $K^\oplus(\Mod_\bF^\aut)$ are connective.
\end{cor}
\begin{proof}
  Since $\NF^\#$ is noetherian, the spectrum $\QK(\NF^\#)$ is connective by
  \cref{rmd:noeth}. Using \cref{lem:split} (or just the fibre sequence from
  \cref{rmd:Sherman}), it follows that also $K^\oplus(\NF)$ is connective. For
  $K^\oplus(\Mod_\bF^\aut)$, we now employ \cref{thm:A}.\looseness-1
\end{proof}

In order to finally determine the homotopy type of $K(\NF^\#)$, we use the
following consequence of dévissage, as in \cite[Cor.\,1]{Quillen}:

\begin{rmd}\label{rmd:schur}
  Let $\bfA$ be an abelian category. We call an object $S\in \bfA$ \emph{simple}
  if $0$ and $S$ are the only subobjects of $S$.  Assume that $\{S_i\}_{i\in I}$
  is a set of representatives for the isomorphism classes of simple objects of
  $\bfA$. If every object $A$ of $\bfA$ has \emph{finite length}, i.e.\ it
  admits a filtration $0=A^0\subset\dotsb\subset A^l= A$ such that each quotient
  $A^{k+1}/A^k$ is simple, then the dévissage theorem provides
  an equivalence of spectra\looseness-1
  \[\QK(\bfA) \simeq \bigoplus_{i\in I}\QK(\Mod_{\on{End}(S_i)^\op}),\]%
  where $\on{End}(S_i)$ is the ring of endomorphisms of $S_i$ (which, by Schur’s
  lemma, is actually a skew field).
\end{rmd}

It is hence our aim to classify the simple objects of $\NF^\#$ and to understand
their endomorphism rings. We start with a general observation.

\begin{rmk}\label{rmk:lift}
  Each additive functor $F\colon (\NF)^\op\to \mathbf{Ab}$ has a
  canonical lift to an $\bF[t]$-linear functor
  $\tilde F\colon(\NF)^\op\kern-.5px\to\kern-.5px
  \Mod_{\smash{\bF[t]}}$ along the forgetful map
  $\Mod_{\smash{\bF[t]}}\kern-1px\to\kern-1px\mathbf{Ab}$ as follows: Given $M$
  in $\NF$, we define an $\bF[t]$-module structure on $FM$ by
  defining the scalar multiplication by $q\in\bF[t]$ to be the map
  $F(M\to M, x\mapsto qx)$. Similarly, a natural transformation
  $\lambda\colon F\to F'$ between additive functors $F$ and $F'$ gives a natural
  transformation $\tilde\lambda\colon \tilde F\to \tilde F'$ between the
  corresponding lifts $\tF$ and $\tF'$.\looseness-1

  By abuse of notation, we will henceforth regard $\NF^\#$ as a subcategory of the
  category of $\bF[t]$-linear functors $\NF^\op\to \Mod_{\smash{\bF[t]}}$.
\end{rmk}

\begin{defi}
  For $r\ge 1$ we define $F_r$ to be the cokernel of $[-,\beta_r]$, where
  \[\beta_r\colon M_{r-1}\oplus M_{r+1}\tto M_r,\quad\quad
    (x,y)\mapsto tx+y,\]%
\end{defi}

\begin{lem}\label{lem:FrMin}
  For all $r,s\ge 1$, we have isomorphisms of $\bF[t]$-modules
  \[F_rM_s\cong \begin{cases}M_1 & \text{for $r=s$,}\\0 &
      \text{else.}\end{cases}\]%
\end{lem}
\begin{proof}
  We note that $\beta_r|_{M_{r-1}} = \phi^{r,r-1}_t$ and
  $\beta_r|_{M_{r+1}} = \phi^{r,r+1}_1$, and moreover:
  \begin{itemize}[itemsep=0em]
  \item if $s<r$, then $[-,\phi^{r,r-1}_t]\colon [M_s,M_{r-1}]\to [M_s,M_r]$ is
    surjective,
  \item if $s>r$, then $[-,\phi^{r,r+1}_1]\colon [M_s,M_{r+1}]\to [M_s,M_r]$ is
    surjective.
  \end{itemize}
  Finally, the evaluation $\on{ev}\colon [M_r,M_r]\to M_r$ with
  $\on{ev}(\phi)=\phi(1)$ is an isomorphism, and we have
  $\on{ev}(\on{im}([M_r,\beta_r])) = tM_r$, and hence
  $M_r/tM_r\cong M_1$.\looseness-1
\end{proof}

\begin{lem}\label{lem:Fi}
  For each $F\ne 0$ in $\NF^\#$, there is an $r\ge 1$ and a monomorphism
  $F_r\hookrightarrow F$.
\end{lem}
\begin{proof}
  Let $r\ge1$ be minimal such that $F\phi^{r,r+1}_1\colon FM_r\to FM_{r+1}$ is
  not injective: such an $r$ exists because $F\ne0$ and
  because of \cref{lem:finite}. Moreover, we pick an element
  $v\in \ker(F\phi_1^{r,r+1})\subseteq FM_r$ with $v\ne 0$. We can additionally
  achieve that $t\cdot v=0$, by repeatedly replacing $v$ by $t\cdot v$ as long
  as this condition is not satisfied.\looseness-1

  The Yoneda lemma gives us a natural transformation
  $Y(v)\colon [-,M_r]\to F$ that sends $\Id_{\smash{M_r}}\in[M_r,M_r]$ to $v$.
  The composition $Y(v)\circ[-,\phi^{r,r+1}_1]\colon [-,M_{r+1}]\to F$ is the
  map $Y((F\phi^{r,r+1}_1)(v))$, which is trivial since $v$ lies in the kernel
  of $F\phi^{r,r+1}_1$. In a similar way, the composition
  $Y(v)\circ[-,\phi^{r,r-1}_t]\colon [-,M_{r-1}]\to F$ is the map
  $Y((F\phi^{r,r-1}_t)(v))$; and again we have $(F\phi^{r,r-1}_t)(v)=0$, as can
  be checked by composing with the injective map $F\phi^{r-1,r}_1$ to get
  $F(\phi^{r-1,r}_1\circ \phi^{r,r-1}_t)(v)=F(t\cdot \on{id}_{M_r})(v)=t\cdot v=
  0$.

  Combining the two previous observations, we obtain that $Y(v)\circ[-,\beta_r]$
  vanishes, whence we get an induced map $\lambda_v\colon F_r\to F$. This
  natural transformation is injective when evaluated on $M_r$, as it gives the
  map $M_1\to FM_r$ sending the standard generator to $v$; it is also trivially
  injective when evaluated on the other objects $M_s$, and by additivity it is
  injective on every object, i.e.\ it is a monomorphism in $\NF^\#$.
\end{proof}

\begin{cor}\label{cor:simple}
  An object $F\in\NF^\#$ is simple if and only if there is an $r\ge 1$ with
  $F\cong F_r$.
\end{cor}
\begin{proof}
  The ‘if’-part is implied by \cref{lem:FrMin} and the fact that $M_1$ has no
  non-trivial submodules, while the ‘only if’-part is implied by
  \cref{lem:Fi}.
\end{proof}

Now we have everything together to prove \cref{thm:B}:

\begin{proof}[Proof of \cref{thm:B}]
  The combination of \cref{lem:Fi} and \cref{lem:finite} shows that each $F$ in
  $\NF^\#$ has finite length. Moreover, by
  \cref{cor:simple}, $\{F_r\}_{r\ge 1}$ is a system of representatives for the
  isomorphism classes of simple objects of $\NF^\#$.  Finally, we note that
  endomorphisms $\eta:F_r\to F_r$ are uniquely determined by their component
  $\eta_{M_r}\colon FM_r\to FM_r$, which is an $\bF$-linear map $\bF\to
  \bF$. Thus, the endomorphism ring of each $F_r$ is isomorphic to
  $\bF=\bF^\op$. By applying \cref{rmd:schur}, we get\looseness-1
  \[\QK(\NF^\#)\simeq \bigoplus_{r\ge 1}\QK(\bF).\]
  Now \cref{thm:B} follows from \cref{lem:split}.
\end{proof}

%% file: 5-topological.tex
The goal of this section is to prove \cref{thm:C}. We start by recalling the
Quillen plus-construction and the group-completion theorem in the language of
$\infty$-categories.

\begin{rmd}[\cite{Hoyois}, {\cite[Prop.\,\textsc{iii}.13]{HebestreitWagner}}, {\cite[Constr.\,3.2.18]{HilmanMcCandless}}]
  A discrete group $G$ is said to be \emph{hypoabelian} if for every non-trivial
  subgroup $H\subset G$, the abelianisation of $H$ is non-trivial. A space $X$
  is \emph{hypoabelian} if the fundamental group $\pi_1(X,*)$ is hypoabelian for
  every choice of basepoint $*\in X$.\looseness-1

  We have a full subcategory $\caS_{\smash\hypo}\subset\caS$ spanned by
  hypoabelian spaces. Moreover, the inclusion functor
  $\imath\colon\caS_{\smash\hypo}\hookrightarrow\caS$ admits a left-adjoint
  functor $(-)^+\colon\caS\to\caS_{\smash\hypo}$, called the \emph{Quillen
    plus-construction}.  If $X\in\caS$ is a hypoabelian space, then the unit of
  the adjunction $X\to \imath(X^+)$ is an equivalence.
\end{rmd}

\begin{rmd}
  The $\infty$-category $\CMon(\caS)$ of $E_\infty$-algebras in spaces contains
  a full subcategory $\CMon(\caS)^\gl$ that is spanned by group-like objects,
  i.e.\ objects $A$ such that the commutative monoid $\pi_0(A)$ is in fact an
  abelian group. The inclusion of $\CMon(\caS)^\gl$ into $\CMon(\caS)$ has a
  left-adjoint, the \emph{group-completion} functor, which agrees with the
  composite $\Omega^\infty \B^\infty$. For each object $A$ in $\CMon(\caS)$, the
  underlying space of $\Omega^\infty \B^\infty A$ has abelian fundamental group
  for each choice of basepoint, and is in particular hypoabelian.
 
  For any $A$ in $\CMon(\caS)$ and any $x\in A$, we define the \emph{mapping telescope}
  $\Tel_xA$ to be the colimit in $\caS$ of the diagram
  $\smash{(A\oset[2px]{-\cdot x\;}{\longrightarrow} A\oset[2px]{-\cdot
      x\;}{\longrightarrow} A\oset[2px]{-\cdot x\;}{\longrightarrow}
    \dotsb)}$. We have a canonical map of spaces
  $\Tel_xA\to\Omega^\infty\B^\infty A$, and since the target is hypoabelian,
  this induces a canonical map of hypoabelian spaces
  $(\Tel_xA)^+\to\Omega^\infty\B^\infty A$.
 
  We say that $x\in A$ is a \emph{propagator} if, for every $y\in A$, there
  exists $z\in A$ and $k\ge0$ such that $x^k$ and $yz$ are in the same
  path-component of $A$. The group-completion theorem
  \cite{SegalMcDuff,RW,Nikolausgroupcompletion} asserts that if $x$ is a
  propagator for $A$, then the canonical map
  $(\Tel_xA)^+\to\Omega^\infty\B^\infty A$ is an equivalence of spaces.
\end{rmd}

\begin{proof}[Proof of \cref{thm:C}]
  In the first part of the proof, we use the symbols “$\mO_\C$” and “$\mO_\R$”
  as aliases of “$\mathrm{U}$” and “$\mO$”, respectively, and we let $\bF$ be
  $\C$ or $\R$. We let $X\in\caS^\omega$ be a compact anima, e.g.\ the
  underlying homotopy type of a finite cell complex.

  We note that $A_\bF\coloneqq\Mod^{\simeq}_{\bF,\top}$ is equivalent, as a
  space, to the disjoint union of classifying spaces
  $\coprod_{d\ge 0}\B\mO_\bF(d)$. Let $x\in A_\bF$ be a point in $\B\mO_\bF(1)$
  and let $c_x\colon X\to A_\bF$ be the constant map with value $x$. Then we
  have a commutative diagram of
  spaces\looseness-1
  \[
    \begin{tikzcd}
      \Tel_{c_x}\Map(X,A_\bF)\ar[r]\ar[d]& (\Tel_{c_x}\Map(X,A_\bF))^+\ar[d]\ar[r]& \Omega^\infty\B^\infty\Map(X,A_\bF)\ar[d]\\
      \Map(X,\Tel_xA_\bF)\ar[r]&\Map(X,(\Tel_xA_\bF)^+)\ar[r] & \Map(X,\Omega^\infty\B^\infty A_\bF).
    \end{tikzcd}
  \]
  Our goal is to prove that the right vertical map is an equivalence of spaces,
  and we will do so by showing that every horizontal map and the left vertical
  map are equivalences.  For the left vertical map, we use that $X$ is a compact space, whence $\Map(X,-)$ commutes with filtered colimits.
  
  For the rightmost horizontal maps, we invoke the group-completion theorem,
  using that $x$ and $c_x$ are propagators in $A_\bF$ and in
  $\Map(X,A_\bF)$, respectively; to justify the latter, note that
  a point $y\in \Map(X,A_\bF)$, i.e. a map
  $y\colon X\to A_\bF\simeq\coprod_{d\ge 0}\B\mO_\bF(d)$, classifies a
  $\bF$-vector bundle over $X$ of some rank $d$; if we let
  $z\colon X\to \B\mO_\bF(d')$ be a map classifying a complementary vector
  bundle over $X$ (i.e. the direct sum of the two vector bundles is a trivial
  vector bundle of rank $d+d'$), then we have $yz\simeq c_x^{d+d'}$.

  For the left horizontal maps, we use that $\Tel_xA_\bF$ and
  $\Tel_{c_x}\Map(X,A_\bF)$ are already hypoabelian, which we see as follows:
  First, $\Tel_xA_\bF$ is equivalent to $\Z\times\B\mO_\bF$, and this space has an
  abelian fundamental group for each choice of basepoint; and second, we see
  $\Tel_{c_x}\Map(X,A_\bF)\simeq\Map(X,\Tel_xA_\bF)\simeq\lim_{\smash{x\in X}}\Tel_xA_\bF$,
  and the full subcategory $\caS_{\smash\hypo}\subset\caS$ is closed under limits. 
  
  This altogether shows that the right vertical map is an equivalence of spaces,
  and since it additionally is a morphism of $E_\infty$-algebras, this concludes
  the proof of the first statement of \cref{thm:C}.\looseness-1
  
  For the second statements about the case of $X=S^1$, let us denote by
  $E\ula{k}$ the $k$-connective cover of a given spectrum $E$. Then we use Bott
  periodicity for complex $K$-theory to obtain and equivalence of spectra
  \[\Map(S^1,\Z\times \BU)\ula{0}\simeq%
    \mathrm{ku}\oplus(\Sigma^{-1}\mathrm{ku})\ula{0}\simeq%
    \mathrm{ku}\oplus\Sigma\mathrm{ku}\simeq %
    \Sigma_+^\infty S^1\wedge \mathrm{ku},\]%
  and we use the Wood cofibre sequence for the ‘realification’ \cite{Wolbert},
  which can be written as
  $\smash{\mathrm{ko}\ula{1}\to
    \Sigma^2\mathrm{ku}\oset[3px]{\Sigma{}^2r}{\longrightarrow}
    \Sigma^2\mathrm{ko}}$, implying that
  $(\Sigma^{-1}\mathrm{ko})\ula{0}\simeq \Sigma^{-1}(\mathrm{ko}\ula{1})$ is the
  cofibre of $r$, to obtain an equivalence of spectra\looseness-1
  \begin{align*}
    \Map(S^1,\Z\times \BO)\ula{0}&\simeq \mathrm{ko}\oplus(\Sigma^{-1}\mathrm{ko})\ula{0}\\
                                 &\simeq\mathrm{ko}\oplus \mathrm{cofib}(r)\\
                                 &\simeq\mathrm{colim}\mathopen{}\big(\mathrm{ko}\overset{r}{\longleftarrow}\mathrm{ku}\overset{r}{\longrightarrow}\mathrm{ko}\big)\mathclose{}.\qedhere
  \end{align*}
\end{proof}